\documentclass[11pt]{article}
\usepackage{color}
\usepackage{epsfig}
\usepackage{subfig}
\usepackage{mathrsfs}
\usepackage{booktabs}
\usepackage{graphicx}
\usepackage{epstopdf}
\usepackage{multirow}
\usepackage{caption}
\usepackage{amsthm}
\usepackage{fancyhdr}
\usepackage{tabularx}
\usepackage{xfrac}
\usepackage{amsmath,amssymb}

\newtheorem{theorem}{Theorem}[section]
\newtheorem{assum}{Assumption}[section]
\newtheorem{defn}{Definition}[section]
\newtheorem{coro}{Corollary}[section]

\newtheorem{lemma}{Lemma}[section]
\newtheorem{remark}{Remark}[section]

\begin{document}

\title{Distributed Sweep Coverage Algorithm of Multi-agent Systems Using Workload Memory}
\author{Chao Zhai \thanks{Chao Zhai is with Future Resilient Systems, Singapore-ETH Centre, ETH Zurich, 1 Create Way, CREATE Tower, Singapore 138602. He is also with Institute of Catastrophe Risk Management, Nanyang Technological University, 50 Nanyang Avenue, Singapore 639798. Email: zhaichao@amss.ac.cn}}

\maketitle

\begin{abstract}
This paper addresses the sweep coverage problem of multi-agent systems in uncertain regions. A new formulation of distributed sweep coverage is proposed to cooperatively complete the workload in the uncertain region. Specifically, each agent takes part in partitioning the whole region while sweeping its own subregion. In addition, the partition operation is carried out to balance the workload in subregions. The trajectories of partition points of agents form the boundaries between adjacent sub-regions. Moreover, it is proved that multi-agent system with the proposed control algorithm is input-to-state stable. Theoretical analysis is conducted to obtain the upper bound of the error between the actual sweep time and the optimal sweep time. Finally, numerical simulations demonstrate the effectiveness of the proposed approach.
\end{abstract}


\section{Introduction}
The great progress in communication technologies makes it easy and costless to share mutual information and coordinate the joint actions among multiple agents. Thus, cooperative control of multi-agent systems has attracted much interest of researchers in various fields in the past decade. The coordination of multiple agents contributes to improving the efficiency and robustness of carrying out the complicated tasks, such as leader tracking \cite{hong06}, flocking behavior \cite{zhang11}, boundary patrolling \cite{albert12,chen11}, persistent monitoring \cite{hussein07,song14}, region coverage \cite{zhai12,zhai13} and even missile interception \cite{zhai16}, to name just a few.

As a type of coordination tasks, cooperative coverage of multi-agent systems refers to the path planning of robot team to visit every point in the environment or the optimal deployment of sensor networks according to the certain performance index. The approach of divide-and-conquer is widely applied in the region coverage of multi-agent systems \cite{chos01,cort04,luna10}. Specifically, \cite{cort04} presents a gradient descent algorithm to optimize a class of utility functions in the coverage region, where the centroidal Voronoi partition is adopted to allocate a subregion for each mobile sensor. \cite{luna10} extends the above work by proposing a distributed, adaptive coverage algorithm for nonholonomic mobile sensors. In multi-robot coordination, the coverage problem falls into three categories: blanket coverage, barrier coverage and sweep coverage\cite{gage92}. Blanket coverage aims
at deploying multiple agents in the given coverage region to maximize the probability of identifying the target \cite{ghosh08}. Barrier coverage is used to protect the target in the given region and meanwhile maximize the detection rate of invaders overpassing the barrier that is formed by the agents \cite{kumar07}. Actually, sweep coverage can be regarded as a moving barrier, and it focuses on sweeping or monitoring the given region by arriving at every point. For example, \cite{cheng09} investigates the sweep coverage of mobile sensors in a corridor environment using the idea of nearest neighbor rules. In addition, \cite{zhai13} develops a decentralized control algorithm for the sweep coverage in uncertain region. To be specific, the whole region is divided into a series of strips, and the agents cooperatively sweep the current stripe while partitioning the adjacent next strip. In theory, the authors provide the upper bound of the error between the optimal sweep time and the actual sweep time.

In this paper, we address the distributed sweep coverage problem of multi-agent systems in uncertain environment. Instead of separating the whole region into several stripes, the agents directly partition the whole region using the trajectories of virtual partition bars while sweeping their respective subregion. We will estimate the error between the optimal sweep time and the actual time using the proposed sweep strategy. The main contributions of this work are listed as follows.
\begin{enumerate}
  \item We develop a new sweep coverage formulation for multi-agent systems using workload memory on the swept region, which overcomes the error accumulation and thus provides a more precise estimation of the error between the actual coverage time and the optimal time.
  \item The sweep coverage algorithm is implemented in the distributed manner, which ensures the robustness of multi-agent systems. Moreover, we provide the sufficient condition for collision avoidance of subregion boundaries in theory.
  \item Compared with previous work \cite{zhai13,zhai12}, the agents in this work need not synchronize their joint action of moving towards the next strip after completing the workload on the current strip, which greatly reduces communication costs among the agents.
\end{enumerate}

The outline of this paper is given as follows. Section \ref{sec:prob} formulates the problem of distributed sweep coverage and develops the control algorithm of multi-agent systems. Section \ref{sec:result} presents theoretical results on the proposed control algorithm, followed by numerical simulations in Section \ref{sec:simu}. Finally, we draw a conclusion and discuss the future direction in Section \ref{sec:con}.

\section{Problem Formulation}\label{sec:prob}
\begin{figure}
\scalebox{0.07}[0.07]{\includegraphics{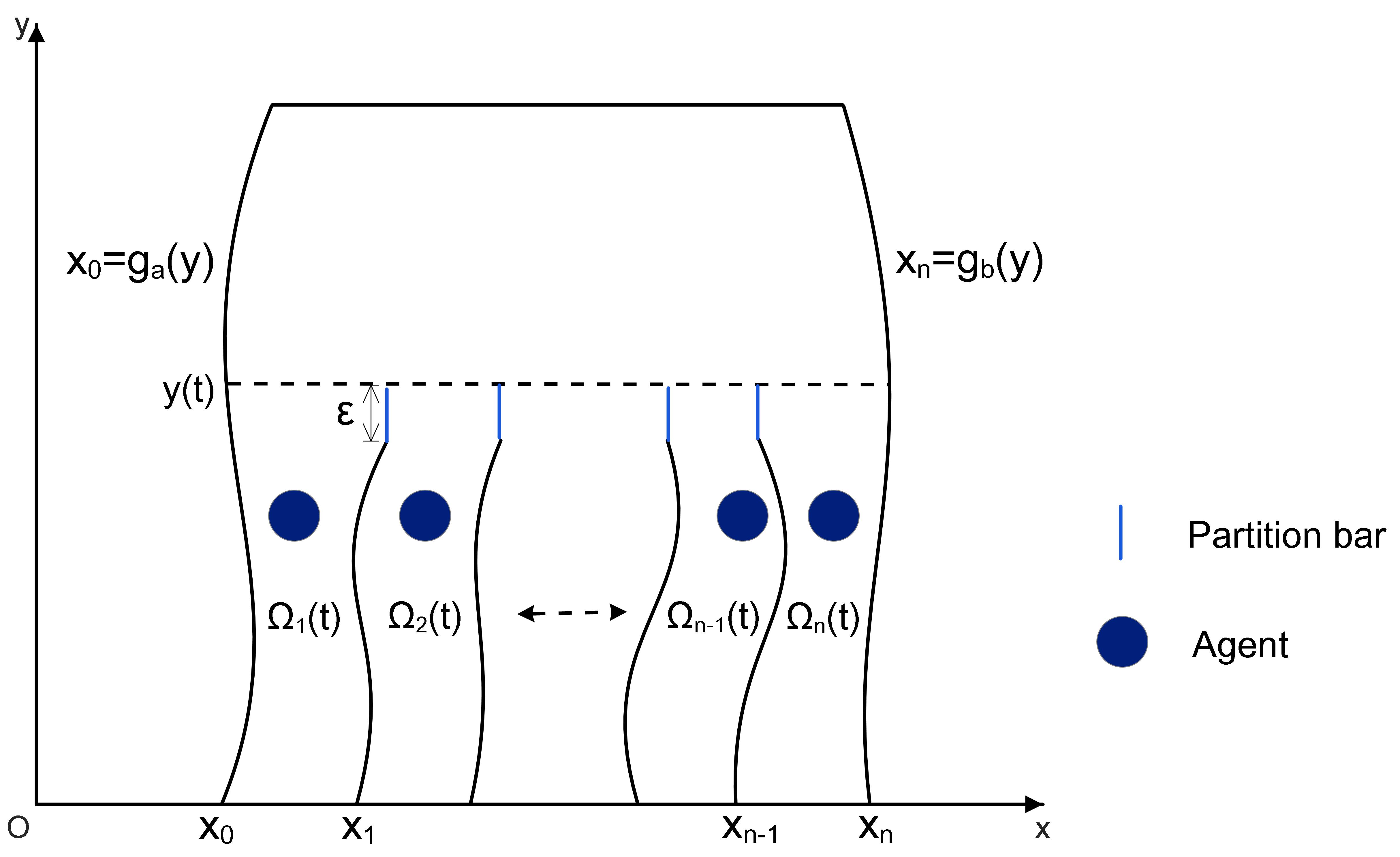}}\centering
\caption{\label{ireg} Sweep coverage of the irregular region with smooth boundaries. The left and right boundaries of coverage region are described by $x_0=g_a(y)$ and $x_{n}=g_b(y)$, respectively.}
\end{figure}
Consider an irregular region $\Omega$ enclosed by two parallel boundaries with the distance $l$ and two smooth boundaries described by $x_0=g_a(y)$ and $x_n=g_b(y)$, respectively (see Fig. \ref{ireg}). There are $n$ mobile agents responsible for sweeping the workload in the region $\Omega$, and the distribution density function of workload is given by $\rho(x,y)$. Each agent is equipped with a virtual partition bar with the length $\epsilon$ to allocate the workload in the whole region, and the trajectories of partition bars form the boundaries between adjacent subregions. Eventually, the partition bars of mobile agents will divide $\Omega$ into $n$ subregions denoted by $\Omega_i,i\in \mathcal{I}_n=\{1,2,...,n\}$. Let $(x_i,y_i)^T\in R^2$, $i\in \mathcal{I}_n$ refer to the coordinate position for the lower terminal point of partition bar for agent $i$,
then the dynamical model of partition bar for agent $i$ is given by
\begin{equation}\label{sys}
\left\{
  \begin{array}{ll}
   \dot{x}_i=u_i(t) \\
    \dot{y}_i=v , & \hbox{$i\in \mathcal{I}_n$.}
  \end{array}
\right.
\end{equation}
where $u_i(t)$ denotes the time-varying control input for agent $i$, and $v$ refers to the constant speed of agents along the $y$-axis. Thus, the complete time of region partition is computed by $T_p=(l-\epsilon)/v$. Each agent can only communicate with its nearest neighbors to share the workload information in their respective subregions. Moreover, the agents are required to evenly partition the workload in the region $\Omega$, and meanwhile they need to complete the workload in their respective subregions with the constant sweeping rate $\sigma$. Each mobile agent can moves back and forth using the boustrophedon ($i.e.$, push-pull ploughing) approach to sweep its own subregion \cite{chos01}. Let $\rho(x,y)$ denote the distribution density function of workload in the region $\Omega$. For $\rho(x,y)$, we have the following assumption \cite{zhai13}
\begin{assum}
There exist positive constants $\bar{\rho}$ and $\underline{\rho}$ such that
$$
\underline{\rho} \leq \rho(x,y)\leq \bar{\rho}, \quad \forall (x,y) \in \Omega
$$
\end{assum}
Let $T^*$ and $T$ represent the optimal sweep time and the actual sweep time using the proposed sweep scheme, respectively.
Thus, $\Delta T=T-T^*$ refers to the corresponding time error. Our goal is to design the control inputs $u_i(t),i\in \mathcal{I}_n$ so that multi-agent systems are able to complete the workload in the region $\Omega$ as soon as possible and estimate the upper bound of time error $\Delta T$. Specifically, the control input $u_i(t)$ is designed as
\begin{equation}\label{control}
u_i(t)=\kappa\left(m_{i+1}(t)-m_{i}(t)\right), \quad i\in \mathcal{I}_{n-1}
\end{equation}
where
$$
m_i(t)=\iint_{\Omega_i(t)}\rho(x,y)dxdy=\int_{0}^{vt+\epsilon}\int_{x_{i-1}}^{x_i}\rho(x,y)dxdy
$$
represents the workload in Subregion $\Omega_i(t)$, and $\kappa$ refers to the positive constant. For agent $n$, we have $\dot{x}_n=g^{'}_b(y)\dot{y}_n=g^{'}_b(vt)v$, since its partition bar moves along the boundary $g_b(y)$ with the vertical speed $v$. For agent $i\in \mathcal{I}_n$, the distributed sweep coverage algorithm (DSCA) is presented in Table~\ref{dsca}. First of all, it is necessary to set the parameters $v$, $\sigma$ and $\epsilon$ for sweep coverage and compute the complete time of region partition $T_p$. Then each agent starts to sweep its own subregion and simultaneously cooperates with other agents for the region partition. After the region partition is completed at $t=T_p$, each agent only focuses on sweeping the workload in its own subregion. Finally, the sweeping task is finished after all the agents clean up the workload in their respective subregions.
\begin{table}
 \caption{\label{dsca}Distributed Sweep Coverage Algorithm.}
 \begin{center}
 \begin{tabular}{lcl} \hline
  1: Set $v$, $\sigma$ and $\epsilon$, and compute $T_p$ \\
  2: \textbf{while}~($e_i(t)>0$) \\
  4: ~~~~~~~\textbf{if}~($t\leq T_p$)  \\
  5: ~~~~~~~~~~Compute the workload $m_i(t)$ in $\Omega_i(t)$ \\
  6: ~~~~~~~~~~Obtain the workload $m_{i+1}(t)$ in $\Omega_{i+1}(t)$ \\
  7: ~~~~~~~~~~Partition the region $\Omega$ with (\ref{sys}) and (\ref{control}) \\
  8: ~~~~~~~\textbf{end if}  \\
  9: ~~~~~~~Sweep the workload in $\Omega_i(t)$ with the rate $\sigma$ \\
  3: ~~~~~~~Compute residual workload $e_i(t)=m_i(t)-\sigma t$ \\
 10: \textbf{end while} \\ \hline
 \end{tabular}
 \end{center}
\end{table}

To facilitate theoretical analysis in Section \ref{sec:result}, it is necessary to provide the concept of input-to-state stability (ISS) of dynamical system \cite{khal}. Thus, we introduce the class $\mathcal{K}$ and $\mathcal{KL}$ functions and then present the definition of ISS as follows.
\begin{defn}
A continuous function $\alpha: [0, a)\rightarrow[0,\infty)$ is said to belong to class $\mathcal{K}$
if it is strictly increasing and $\alpha(0)=0$.
\end{defn}

\begin{defn}
A continuous function $\beta:[0, a)\times[0,\infty)\rightarrow[0,\infty)$ is said to belong
to class $\mathcal{KL}$ if for each fixed $s$, the mapping $\beta(r, s)$ belongs to class $K$ with respect
to $r$ and, for each fixed $r$, the mapping $\beta(r, s)$ is decreasing with respect to $s$ and
$\beta(r,s)\rightarrow0$ as $s\rightarrow\infty$.
\end{defn}

\begin{defn}\label{def_iss}
The system $\dot{z}=f(z,t,u)$ is said to be input-to-state stable if there exist a class $\mathcal{KL}$ function $\beta$ and a class $\mathcal{K}$ function $\alpha$ such that for any initial state $z(t_0)$ and any bounded input $u(t)$, the solution $z(t)$ exists for all $t\geq t_0$ and satisfies
\begin{equation}\label{iss}
\|z(t)\|\leq \beta(\|z(t_0)\|,t-t_0)+\alpha\left(\sup_{t_0\leq\tau\leq t}\|u(\tau)\|\right)
\end{equation}
\end{defn}

\section{Main Results}\label{sec:result}

To quantify the partition performance of multi-agent system, we introduce the following energy function
\begin{equation}\label{energy}
    H(t)=\sum_{i=1}^{n}\left(m_i(t)-\bar{m}(t)\right)^2
\end{equation}
where
$$
\bar{m}(t)=\frac{1}{n}\iint_{\Omega(t)}\rho(x,y)dxdy=\frac{1}{n}\int_{0}^{vt+\epsilon}\int_{x_{0}}^{x_{n}}\rho(x,y)dxdy
$$
Next, we present four lemmas in order to prove the ISS of multi-agent system (\ref{sys}).
\begin{lemma}\label{ineq}
$$
\frac{\lambda_{min}}{n}H(t)\leq\sum_{i=1}^{n-1}\left(m_{i+1}(t)-m_{i}(t)\right)^2\leq\lambda_{max}H(t),
$$
where $\lambda_{min}$ and $\lambda_{max}$ are respectively the minimum and maximum eigenvalues of a positive definite matrix $\Gamma$ that depends on the communication topology of multi-agent systems.
\end{lemma}
\begin{proof}
The proof follows by the same argument as that in Lemma 3.2 in \cite{zhai12}, and is thus omitted.
\end{proof}

\begin{lemma}\label{ineq2}
$$
\left(\sum_{i=1}^{n}\left(m_i(t)-\bar{m}(t)\right)^2\right)^{\frac{1}{2}}\leq \sqrt{n(n-1)}\bar{m}(t)
$$
\end{lemma}
\begin{proof}
It follows from
\begin{equation*}
\begin{split}
\sum_{i=1}^{n}\left(m_i(t)-\bar{m}(t)\right)^2&=\sum_{i=1}^{n}m_i(t)^2+n\bar{m}(t)^2-2\bar{m}(t)\sum_{i=1}^{n}m_i(t)\\
&=\sum_{i=1}^{n}m_i(t)^2-n\bar{m}(t)^2 \\
&\leq \left(\sum_{i=1}^{n}m_i(t)\right)^2-n\bar{m}(t)^2 \\
&=n(n-1)\bar{m}(t)^2 \\
\end{split}
\end{equation*}
that
$$
\left(\sum_{i=1}^{n}\left(m_i(t)-\bar{m}(t)\right)^2\right)^{\frac{1}{2}}\leq \sqrt{n(n-1)}\bar{m}(t),
$$
which completes the proof.
\end{proof}

\begin{lemma}\label{ineq3}
$$
\max_{i\in \mathcal{I}_n}|m_i(t)-\bar{m}(t)|\leq\left(\frac{n-1}{n}\sum_{i=1}^{n}\left(m_i(t)-\bar{m}(t)\right)^2\right)^{\frac{1}{2}}
$$
\end{lemma}

\begin{proof}
First of all, we split $H(t)$ into two terms as follows
\begin{equation}\label{split}
    \sum_{i=1}^{n}\left(m_i(t)-\bar{m}(t)\right)^2=\left(m_k(t)-\bar{m}(t)\right)^2+\sum_{i=1,i\neq k}^{n}\left(m_i(t)-\bar{m}(t)\right)^2 \quad \forall k\in \mathcal{I}_n
\end{equation}
Next, define $z_i(t)=m_i(t)-\bar{m}(t)$, and we formulate the constrained optimization problem
\begin{equation}\label{op}
    \min\sum_{i=1,i\neq k}^{n}z_i(t)^2,
\end{equation}
which is subject to $\sum_{i=1,i\neq k}^{n}z_i(t)=-z_k(t)$, since we have $\sum_{i=1}^{n}z_i(t)=0$. To solve optimization problem (\ref{op}), we introduce the Lagrange function
$$
\mathcal{L}(z_1,z_2,...,z_{n},c)=\sum_{i=1,i\neq k}^{n}z_i(t)^2-c\sum_{i=1}^{n}z_i(t)
$$
By solving the system of equation
\begin{equation*}
\begin{split}
\frac{\partial \mathcal{L}}{\partial z_i}&=0, \quad i\in \mathcal{I}_n, i\neq k\\
\frac{\partial \mathcal{L}}{\partial c}&=\sum_{i=1}^{n}z_i(t)=0
\end{split}
\end{equation*}
we get $z_i(t)=-\frac{z_k(t)}{n-1}, i\in \mathcal{I}_n, i\neq k$ and the minimum of optimization problem (\ref{op})
$$
 \min\sum_{i=1,i\neq k}^{n+1}z_i(t)^2=\frac{z_k(t)^2}{n-1},
$$
which implies
$$
\sum_{i=1,i\neq k}^{n}\left(m_i(t)-\bar{m}(t)\right)^2\geq\frac{1}{n-1}\left(m_k(t)-\bar{m}(t)\right)^2
$$
Then it follows from equation (\ref{split}) that
\begin{equation*}
\begin{split}
    \sum_{i=1}^{n}\left(m_i(t)-\bar{m}(t)\right)^2&\geq\left(m_k(t)-\bar{m}(t)\right)^2+\frac{1}{n-1}\left(m_k(t)-\bar{m}(t)\right)^2\\
    &= \frac{n}{n-1}\left(m_k(t)-\bar{m}(t)\right)^2, \quad \forall k\in \mathcal{I}_n
\end{split}
\end{equation*}
which indicates
$$
|m_k(t)-\bar{m}(t)|\leq \left(\frac{n-1}{n}\sum_{i=1}^{n}\left(m_i(t)-\bar{m}(t)\right)^2\right)^{\frac{1}{2}}, \quad \forall k\in \mathcal{I}_n
$$
Therefore, we conclude
$$
\max_{i\in \mathcal{I}_n}|m_i(t)-\bar{m}(t)|\leq\left(\frac{n-1}{n}\sum_{i=1}^{n}\left(m_i(t)-\bar{m}(t)\right)^2\right)^{\frac{1}{2}}
$$
The proof is thus completed.
\end{proof}

\begin{lemma}\label{ineq4}
The solution to the following differential inequality
\begin{equation}\label{dHt_ine}
\frac{dH(t)}{dt}\leq -\xi(t) H(t)+\zeta(t)\sqrt{H(t)}
\end{equation}
gives
$$
\sqrt{H(t)}\leq \sqrt{H(0)}~e^{-\frac{1}{2}\int_{0}^{t}\xi(\tau)d\tau}+\frac{1}{2}\int_{0}^{t}e^{\frac{1}{2}\int_{t}^{\tau}\xi(s)ds}\zeta(\tau)d\tau
$$
where $\xi(t)$ and $\zeta(t)$ are non-negative continuous functions with respect to time $t$.
\end{lemma}

\begin{proof}
Since $H(t)\geq0$, we have
$$
\frac{dH(t)}{dt}=\frac{d\left(\sqrt{H(t)}\right)^2}{dt}=2\sqrt{H(t)}\cdot\frac{d\sqrt{H(t)}}{dt}.
$$
Then the inequality (\ref{dHt_ine}) can be rewritten as
\begin{equation}\label{sqHt}
2\sqrt{H(t)}\cdot\frac{d\sqrt{H(t)}}{dt}\leq -\xi(t)H(t)+\zeta(t)\sqrt{H(t)}
\end{equation}
For $H(t)\neq0$, the inequality (\ref{sqHt}) can be simplified as
$$
\frac{d\sqrt{H(t)}}{dt}\leq -\frac{\xi(t)}{2}\sqrt{H(t)}+\frac{\zeta(t)}{2}
$$
Then it follows from Comparison Lemma \cite{khal} that
$$
\sqrt{H(t)}\leq \sqrt{H(0)}~e^{-\frac{1}{2}\int_{0}^{t}\xi(\tau)d\tau}+\frac{1}{2}\int_{0}^{t}e^{\frac{1}{2}\int_{t}^{\tau}\xi(s)ds}\zeta(\tau)d\tau
$$
For $H(t)=0$, the above inequality still holds, since both $\xi(t)$ and $\zeta(t)$ are non-negative. This completes the proof.
\end{proof}

The time derivative of $H(t)$ with respect to the trajectory of system (\ref{sys}) is given by
\begin{equation*}
    \frac{dH(t)}{dt}=2\kappa\sum_{i=1}^{n}\left(\frac{dm_i(t)}{dt}-\frac{d\bar{m}(t)}{dt}\right)\left(m_i(t)-\bar{m}(t)\right)
\end{equation*}
where the time derivative of $\bar{m}(t)$ is given by
$$
\frac{d\bar{m}(t)}{dt}=\frac{v}{n}\int_{x_{0}(vt+\epsilon)}^{x_{n}(vt+\epsilon)}\rho(x,vt+\epsilon)dx=\frac{v}{n}\int_{g_{a}(vt+\epsilon)}^{g_{b}(vt+\epsilon)}\rho(x,vt+\epsilon)dx
$$
and
\begin{equation*}
\begin{split}
\frac{dm_i(t)}{dt}
&=\frac{d}{dt}\left(\int_{0}^{vt+\epsilon}\int_{x_{i-1}}^{x_i}\rho(x,y)dxdy\right)\\
&=\frac{d}{dt}\left(\int_{0}^{vt}\int_{x_{i-1}}^{x_i}\rho(x,y)dxdy+\int_{vt}^{vt+\epsilon}\int_{x_{i-1}}^{x_i}\rho(x,y)dxdy\right)\\
&=\frac{d}{dt}\left(\int_{0}^{vt}\int_{x_{i-1}}^{x_i}\rho(x,y)dxdy\right)+\frac{d}{dt}\left(\int_{vt}^{vt+\epsilon}\int_{x_{i-1}}^{x_i}\rho(x,y)dxdy\right)\\
&=v\int_{x_{i-1}(vt+\epsilon)}^{x_{i}(vt+\epsilon)}\rho(x,vt+\epsilon)dx+\int_{vt}^{vt+\epsilon}\frac{\partial}{\partial t}\int_{x_{i-1}}^{x_i}\rho(x,y)dxdy\\
&=v\int_{x_{i-1}(vt+\epsilon)}^{x_{i}(vt+\epsilon)}\rho(x,vt+\epsilon)dx+\int_{vt}^{vt+\epsilon}\left[\rho(x_{i},y)\dot{x}_{i}-\rho(x_{i-1},y)\dot{x}_{i-1}\right]dy\\
&=v\int_{x_{i-1}(vt+\epsilon)}^{x_{i}(vt+\epsilon)}\rho(x,vt+\epsilon)dx+\dot{x}_{i}\int_{vt}^{vt+\epsilon}\rho(x_{i},y)dy-\dot{x}_{i-1}\int_{vt}^{vt+\epsilon}\rho(x_{i-1},y)dy\\
&=v\int_{x_{i-1}(vt+\epsilon)}^{x_{i}(vt+\epsilon)}\rho(x,vt+\epsilon)dx+\kappa\left(m_{i+1}(t)-m_{i}(t)\right)\int_{vt}^{vt+\epsilon}\rho(x_{i},y)dy\\
&-\kappa\left(m_{i}(t)-m_{i-1}(t)\right)\int_{vt}^{vt+\epsilon}\rho(x_{i-1},y)dy
\end{split}
\end{equation*}
In addition, we have
$$
\dot{x}_0=\frac{dg_a(vt+\epsilon)}{dt}=v\frac{dg_a(y)}{dy}|_{y=vt+\epsilon}=vg^{'}_a(vt+\epsilon),  \quad \dot{x}_{n}=\frac{dg_b(vt+\epsilon)}{dt}=v\frac{dg_b(y)}{dy}|_{y=vt+\epsilon}=vg^{'}_b(vt+\epsilon)
$$
and the time derivatives of $m_1(t)$ and $m_{n}(t)$ are presented as
\begin{equation*}
\begin{split}
\frac{dm_1(t)}{dt}&=\frac{d}{dt}\int_{0}^{vt+\epsilon}\int_{g_a(y)}^{x_1}\rho(x,y)dxdy \\
&=v\int_{g_a(vt+\epsilon)}^{x_1(vt+\epsilon)}\rho(x,vt+\epsilon)dx+\dot{x}_{1}\int_{vt}^{vt+\epsilon}\rho(x_{1},y)dy-g^{'}_a(vt+\epsilon)v\int_{vt}^{vt+\epsilon}\rho(g_a(y),y)dy
\end{split}
\end{equation*}
and
\begin{equation*}
\begin{split}
\frac{dm_{n}(t)}{dt}&=\frac{d}{dt}\int_{0}^{vt+\epsilon}\int_{x_{n-1}}^{g_b(y)}\rho(x,y)dxdy \\
&=v\int_{x_{n-1}(vt+\epsilon)}^{g_b(vt+\epsilon)}\rho(x,vt+\epsilon)dx+g^{'}_b(vt+\epsilon)v\int_{vt}^{vt+\epsilon}\rho(g_b(y),y)dy-\dot{x}_{n-1}\int_{vt}^{vt+\epsilon}\rho(x_{n-1},y)dy,
\end{split}
\end{equation*}
respectively. Next, we present theoretical results on the stability of multi-agent system (\ref{sys}) as follows.
\begin{theorem}
Multi-agent system (\ref{sys}) with the control input (\ref{control}) is input-to-state stable.
\end{theorem}

\begin{proof}
The time derivative of $H(t)$ with respect to the trajectory of dynamical system (\ref{sys}) is given by
\begin{equation}\label{dHt}
\begin{split}
\frac{dH(t)}{dt}&=2\sum_{i=1}^{n}\left(m_i(t)-\bar{m}(t)\right)\left(\dot{x}_{i}\int_{vt}^{vt+\epsilon}\rho(x_{i},y)dy-\dot{x}_{i-1}\int_{vt}^{vt+\epsilon}\rho(x_{i-1},y)dy\right)\\
&+2v\sum_{i=1}^{n}\left(m_i(t)-\bar{m}(t)\right)\left(\int_{x_{i-1}(vt+\epsilon)}^{x_{i}(vt+\epsilon)}\rho(x,vt+\epsilon)dx-\frac{1}{n}\int_{g_a(vt+\epsilon)}^{g_b(vt+\epsilon)}\rho(x,vt+\epsilon)dx\right)\\
\end{split}
\end{equation}
The first term in equation (\ref{dHt}) can be further expressed as
\begin{equation*}
\begin{split}
&~~~~\sum_{i=1}^{n}\left(m_i(t)-\bar{m}(t)\right)\left(\dot{x}_{i}\int_{vt}^{vt+\epsilon}\rho(x_{i},y)dy-\dot{x}_{i-1}\int_{vt}^{vt+\epsilon}\rho(x_{i-1},y)dy\right)\\
&=\sum_{i=1}^{n-1}\left(m_i(t)-\bar{m}(t)\right)\dot{x}_{i}\int_{vt}^{vt+\epsilon}\rho(x_{i},y)dy+ \left(m_{n}(t)-\bar{m}(t)\right)\dot{x}_{n}\int_{vt}^{vt+\epsilon}\rho(x_{n},y)dy\\
&-\sum_{i=2}^{n}\left(m_i(t)-\bar{m}(t)\right)\dot{x}_{i-1}\int_{vt}^{vt+\epsilon}\rho(x_{i-1},y)dy-\left(m_1(t)-\bar{m}(t)\right)\dot{x}_{0}\int_{vt}^{vt+\epsilon}\rho(x_{0},y)dy\\
&=\sum_{i=1}^{n-1}\left(m_i(t)-\bar{m}(t)\right)\dot{x}_{i}\int_{vt}^{vt+\epsilon}\rho(x_{i},y)dy
+ v\left(m_{n}(t)-\bar{m}(t)\right)g^{'}_b(vt+\epsilon)\int_{vt}^{vt+\epsilon}\rho(g_b(y),y)dy\\
&-\sum_{i=1}^{n-1}\left(m_{i+1}(t)-\bar{m}(t)\right)\dot{x}_{i}\int_{vt}^{vt+\epsilon}\rho(x_{i},y)dy
- v\left(m_1(t)-\bar{m}(t)\right)g^{'}_a(vt+\epsilon)\int_{vt}^{vt+\epsilon}\rho(g_a(y),y)dy\\
&=-\kappa\sum_{i=1}^{n-1}\left(m_{i+1}(t)-m_i(t)\right)^2\int_{vt}^{vt+\epsilon}\rho(x_{i},y)dy+ v\left(m_{n}(t)-\bar{m}(t)\right)g^{'}_b(vt+\epsilon)\int_{vt}^{vt+\epsilon}\rho(g_b(y),y)dy\\
&- v\left(m_1(t)-\bar{m}(t)\right)g^{'}_a(vt+\epsilon)\int_{vt}^{vt+\epsilon}\rho(g_a(y),y)dy\\
\end{split}
\end{equation*}
Considering that inequalities
$$
-\kappa\sum_{i=1}^{n-1}\left(m_{i+1}(t)-m_i(t)\right)^2\int_{vt}^{vt+\epsilon}\rho(x_{i},y)dy\leq -\kappa\underline{\rho}\epsilon\sum_{i=1}^{n-1}\left(m_{i+1}(t)-m_{i}(t)\right)^2
$$
and
\begin{equation*}
\begin{split}
&~~~~v\left(m_{n}(t)-\bar{m}(t)\right)g^{'}_b(vt+\epsilon)\int_{vt}^{vt+\epsilon}\rho(g_b(y),y)dy \\
&-v\left(m_1(t)-\bar{m}(t)\right)g^{'}_a(vt+\epsilon)\int_{vt}^{vt+\epsilon}\rho(g_a(y),y)dy \\
&\leq 2v\epsilon\bar{\rho}\max_{i\in \mathcal{I}_n}|m_i(t)-\bar{m}(t)|\cdot\max\{|g^{'}_a(vt+\epsilon)|,|g^{'}_b(vt+\epsilon)|\},
\end{split}
\end{equation*}
hold, the first term in equation (\ref{dHt}) can be estimated as follows
\begin{equation*}
\begin{split}
&~~~~\sum_{i=1}^{n}\left(m_i(t)-\bar{m}(t)\right)\left(\dot{x}_{i}\int_{vt}^{vt+\epsilon}\rho(x_{i},y)dy-\dot{x}_{i-1}\int_{vt}^{vt+\epsilon}\rho(x_{i-1},y)dy\right)\\
&\leq-\kappa\underline{\rho}\epsilon\sum_{i=1}^{n-1}\left(m_{i+1}(t)-m_{i}(t)\right)^2+2v\epsilon\bar{\rho}\max_{i\in \mathcal{I}_n}|m_i(t)-\bar{m}(t)|\cdot\max\{|g^{'}_a(vt+\epsilon)|,|g^{'}_b(vt+\epsilon)|\}\\
&=-\frac{\kappa\underline{\rho}\epsilon\lambda_{min}}{n}H(t)+2v\epsilon\bar{\rho}\max_{i\in \mathcal{I}_n}|m_i(t)-\bar{m}(t)|\cdot\max\{|g^{'}_a(vt+\epsilon)|,|g^{'}_b(vt+\epsilon)|\}\\
&\leq -\frac{\kappa\underline{\rho}\epsilon\lambda_{min}}{n}H(t)+2v\epsilon\bar{\rho}\sqrt{\frac{n-1}{n}}\max\{|g^{'}_a(vt+\epsilon)|,|g^{'}_b(vt+\epsilon)|\}\sqrt{H(t)}\\
\end{split}
\end{equation*}
For the second term, it follows from Cauchy-Schwarz inequality and Lemma \ref{ineq2} that
\begin{equation*}
\begin{split}
&~~~~\sum_{i=1}^{n}\left(m_i(t)-\bar{m}(t)\right)\left(\int_{x_{i-1}(vt+\epsilon)}^{x_{i}(vt+\epsilon)}\rho(x,vt+\epsilon)dx-\frac{1}{n}\int_{g_a(vt+\epsilon)}^{g_b(vt+\epsilon)}\rho(x,vt+\epsilon)dx\right)\\
&\leq \left[\sum_{i=1}^{n}\left(m_i(t)-\bar{m}(t)\right)^2\right]^{\frac{1}{2}}\left[\sum_{i=1}^{n}\left(\int_{x_{i-1}(vt+\epsilon)}^{x_{i}(vt+\epsilon)}\rho(x,vt+\epsilon)dx-\frac{1}{n}\int_{g_a(vt+\epsilon)}^{g_b(vt+\epsilon)}\rho(x,vt+\epsilon)dx\right)^2\right]^{\frac{1}{2}}\\
&=\sqrt{H(t)}\left[\sum_{i=1}^{n}\left(\int_{x_{i-1}(vt+\epsilon)}^{x_{i}(vt+\epsilon)}\rho(x,vt+\epsilon)dx-\frac{1}{n}\int_{g_a(vt+\epsilon)}^{g_b(vt+\epsilon)}\rho(x,vt+\epsilon)dx\right)^2\right]^{\frac{1}{2}}\\
&\leq\sqrt{\frac{n-1}{n}}\cdot \bar{\rho}\left[g_b(vt+\epsilon)-g_a(vt+\epsilon)\right]\cdot\sqrt{H(t)}
\end{split}
\end{equation*}
Therefore, the time derivative of $H(t)$ can be estimated as follows
\begin{equation}\label{dhdt}
\begin{split}
\frac{dH(t)}{dt}
&\leq-\frac{\kappa\underline{\rho}\epsilon\lambda_{min}}{n}H(t)+2v\epsilon\bar{\rho}\sqrt{\frac{n-1}{n}}\max\{|g^{'}_a(vt+\epsilon)|,|g^{'}_b(vt+\epsilon)|\}\sqrt{H(t)}\\
&+\sqrt{\frac{n-1}{n}}\cdot \bar{\rho}\left[g_b(vt+\epsilon)-g_a(vt+\epsilon)\right]\cdot\sqrt{H(t)}\\
&=-\frac{\kappa\underline{\rho}\epsilon\lambda_{min}}{n}H(t)\\
&+\bar{\rho}\left(2v\epsilon\max\{|g^{'}_a(vt+\epsilon)|,|g^{'}_b(vt+\epsilon)|\}+g_b(vt+\epsilon)-g_a(vt+\epsilon)\right)\sqrt{\frac{n-1}{n}}\sqrt{H(t)}
\end{split}
\end{equation}
According to Lemma \ref{ineq}, we have
$$
\sqrt{H(t)}\leq \left(\frac{n}{\lambda_{min}} \sum_{i=1}^{n-1}\left(m_{i+1}(t)-m_{i}(t)\right)^2\right)^{\frac{1}{2}}=\frac{1}{\kappa}\sqrt{\frac{n}{\lambda_{min}}}\|u(t)\|_2
$$
where $u(t)=(u_1(t),u_2(t),...,u_{n-1}(t))\in R^{n-1}$. Thus, the inequality (\ref{dhdt}) can be further expressed as
\begin{equation}\label{difinq}
\begin{split}
\frac{dH(t)}{dt}
&\leq-\frac{\kappa\underline{\rho}\epsilon\lambda_{min}}{n}H(t)\\
&+\frac{\bar{\rho}}{\kappa}\left(2v\epsilon\max\{|g^{'}_a(vt+\epsilon)|,|g^{'}_b(vt+\epsilon)|\}+g_b(vt+\epsilon)-g_a(vt+\epsilon)\right)\sqrt{\frac{n-1}{\lambda_{\min}}}\|u(t)\|_2\\
&=-\xi(t)H(t)+\Theta(t)\|u(t)\|_2
\end{split}
\end{equation}
where
\begin{equation}\label{xi}
\xi(t)=\frac{\kappa\underline{\rho}\epsilon\lambda_{min}}{n}
\end{equation}
and
$$
\Theta(t)=\frac{\bar{\rho}}{\kappa}\left(2v\epsilon\max\{|g^{'}_a(vt+\epsilon)|,|g^{'}_b(vt+\epsilon)|\}+g_b(vt+\epsilon)-g_a(vt+\epsilon)\right)\sqrt{\frac{n-1}{\lambda_{\min}}}
$$
Solving the differential inequality (\ref{difinq}) gives
\begin{equation}\label{soldif}
H(t)\leq H(t_0)~e^{-\int_{t_0}^{t}\xi(\tau)d\tau}+\int_{t_0}^{t}e^{\int_{t}^{\tau}\xi(s)ds}\Theta(\tau)\|u(\tau)\|_2d\tau
\end{equation}
Define $ z(t)=\left(z_1(t),z_2(t),...,z_{n}(t)\right)^T\in R^{n}$ with $z_i(t)=m_i(t)-\bar{m}(t),i\in\mathcal{I}_n$. Then we can get $\sqrt{H(t)}=\|z(t)\|_2$. Since $t_0\leq\tau\leq t$ in inequality (\ref{dHt1}), we have $e^{\int_{t}^{\tau}\xi(s)ds}\leq 1$. Thus, the inequality (\ref{soldif}) can be rewritten as
\begin{equation*}
\begin{split}
\|z(t)\|^2_2&\leq\|z(t_0)\|^2_2~e^{-\int_{t_0}^{t}\xi(\tau)d\tau}+\int_{t_0}^{t}e^{\int_{t}^{\tau}\xi(s)ds}\Theta(\tau)\|u(\tau)\|_2d\tau\\
&\leq\|z(t_0)\|^2_2~e^{-\int_{t_0}^{t}\xi(\tau)d\tau}+\int_{t_0}^{t}\Theta(\tau)\|u(\tau)\|_2d\tau \\
\end{split}
\end{equation*}
which is equivalent to
\begin{equation*}
\begin{split}
\|z(t)\|_2&\leq\left(\|z(t_0)\|^2_2~e^{-\int_{t_0}^{t}\xi(\tau)d\tau}+\int_{t_0}^{t}\Theta(\tau)\|u(\tau)\|_2d\tau\right)^{\frac{1}{2}}\\
&\leq\|z(t_0)\|_2~e^{-\frac{1}{2}\int_{t_0}^{t}\xi(\tau)d\tau}+\left(\int_{t_0}^{t}\Theta(\tau)\|u(\tau)\|_2~d\tau \right)^{\frac{1}{2}} \\
\end{split}
\end{equation*}
It follows from Definition \ref{def_iss} that multi-agent system (\ref{sys}) with the control input (\ref{control}) is input-to-state stable. The proof is thus completed.
\end{proof}
Actually, the inequality (\ref{dhdt}) can be rewritten as
$$
\frac{dH(t)}{dt}\leq-\xi(t)H(t)+\zeta(t)\sqrt{H(t)}
$$
where
\begin{equation}\label{zeta}
\zeta(t)=\bar{\rho}\left(2v\epsilon\max\{|g^{'}_a(vt+\epsilon)|,|g^{'}_b(vt+\epsilon)|\}+g_b(vt+\epsilon)-g_a(vt+\epsilon)\right)\sqrt{\frac{n-1}{n}}
\end{equation}
From Lemma \ref{ineq4}, we derive
\begin{equation}\label{dHt1}
\sqrt{H(t)}\leq \sqrt{H(t_0)}e^{-\frac{1}{2}\int_{t_0}^{t}\xi(\tau)d\tau}+\frac{1}{2}\int_{t_0}^{t}e^{\frac{1}{2}\int_{t}^{\tau}\xi(s)ds}\zeta(\tau)d\tau
\end{equation}

Let $\Delta T$ denote the time error between the actual sweeping time and the optimal sweeping time for multi-agent system. Then we can obtain the upper bound of $\Delta T$ in the following theorem.
\begin{theorem}
With the DSCA in Table \ref{dsca}, the time error $\Delta T$ for sweeping the irregular region $\Omega$ is bounded by
$$
\Delta T\leq \frac{1}{\sigma}\sqrt{\frac{n-1}{n}}\left(\sqrt{H(0)}e^{-\frac{\kappa\underline{\rho}\epsilon\lambda_{min}T_p}{2n}}+\frac{T_p}{2q}\sum_{i=1}^{q}e^{-\frac{\kappa\underline{\rho}\epsilon(q-i)\lambda_{min}T_p}{2nq}}\max_{\tau\in c_i}\zeta(\tau)\right)
$$
with $T_p=(l-\epsilon)/v$ and $c_i=[(i-1)T_p/q,iT_p/q]$, where $q\in Z^{+}$ is a positive integer and $\zeta(t)$ is given by equation (\ref{zeta}).
\end{theorem}

\begin{proof}
Considering that the sweeping rate is the same for all agents, $\Delta T$ depends on the subregion with the most workload, so it can be computed as
\begin{equation}
\begin{split}
    \Delta T&=\frac{1}{\sigma}\max_{i\in I_{n}}|m_i(T_p)-\bar{m}(T_p)| \\
\end{split}
\end{equation}
From Lemma \ref{ineq3} and equation (\ref{dHt1}) with $t_0=0$, we have
\begin{equation*}
\begin{split}
    \Delta T&\leq\frac{1}{\sigma}\left(\frac{n-1}{n}\sum_{i=1}^{n+1}\left(m_i(T_p)-\bar{m}(T_p)\right)^2\right)^{\frac{1}{2}}\\
     &=\frac{1}{\sigma}\sqrt{\frac{n-1}{n}}\sqrt{H(T_p)} \\
     &\leq \frac{1}{\sigma}\sqrt{\frac{n-1}{n}}\left(\sqrt{H(0)}e^{-\frac{1}{2}\int_{0}^{T_p}\xi(\tau)d\tau}+\frac{1}{2}\int_{0}^{T_p}e^{\frac{1}{2}\int_{T_p}^{\tau}\xi(s)ds}\zeta(\tau)d\tau\right)\\
\end{split}
\end{equation*}
Divide the time interval $[0,T_p]$ into a series of consecutive subintervals $[t_{i-1},t_i]$ with $t_i=iT_p/q$, $i=1,2,...,q$ and $q\in Z^{+}$. Considering that
$$
e^{\frac{1}{2}\int_{T_p}^{\tau}\xi(s)ds}, \quad 0\leq\tau\leq T_p
$$
is an increasing function with respect to $\tau$, we have
\begin{equation}
\begin{split}
\int_{0}^{T_p}e^{\frac{1}{2}\int_{T_p}^{\tau}\xi(s)ds}\zeta(\tau)d\tau &=\sum_{i=1}^{q}\int_{t_{i-1}}^{t_i}e^{\frac{1}{2}\int_{T_p}^{\tau}\xi(s)ds}\zeta(\tau)d\tau\\
&\leq\sum_{i=1}^{q}e^{-\frac{1}{2}\int_{iT_p/q}^{T_p}\xi(s)ds}\int_{t_{i-1}}^{t_i}\zeta(\tau)d\tau\\
&=\sum_{i=1}^{q}e^{-\frac{\kappa\underline{\rho}\epsilon(q-i)\lambda_{min}T_p}{2nq}}\int_{(i-1)T_p/q}^{iT_p/q}\zeta(\tau)d\tau\\
\end{split}
\end{equation}
Moreover, it follows from equation (\ref{xi}) that
$$
\int_{0}^{T_p}\xi(\tau)d\tau=\int_{0}^{T_p}\frac{\kappa\underline{\rho}\epsilon\lambda_{min}}{n}d\tau=\frac{\kappa\underline{\rho}\epsilon\lambda_{min}T_p}{n}
$$
which leads to
\begin{equation}\label{updTb}
\begin{split}
\Delta T&\leq \frac{1}{\sigma}\sqrt{\frac{n-1}{n}}\left(\sqrt{H(0)}e^{-\frac{\kappa\underline{\rho}\epsilon\lambda_{min}T_p}{2n}}+\frac{1}{2}\sum_{i=1}^{q}e^{-\frac{\kappa\underline{\rho}\epsilon(q-i)\lambda_{min}T_p}{2nq}}\int_{(i-1)T_p/q}^{iT_p/q}\zeta(\tau)d\tau\right)\\
&\leq\frac{1}{\sigma}\sqrt{\frac{n-1}{n}}\left(\sqrt{H(0)}e^{-\frac{\kappa\underline{\rho}\epsilon\lambda_{min}T_p}{2n}}+\frac{T_p}{2q}\sum_{i=1}^{q}e^{-\frac{\kappa\underline{\rho}\epsilon(q-i)\lambda_{min}T_p}{2nq}}\max_{\tau\in c_i}\zeta(\tau)\right)\\
\end{split}
\end{equation}
The proof is thus completed.
\end{proof}

\begin{remark}
Compared with previous work \cite{zhai12,zhai13}, the distributed sweep coverage algorithm in Table \ref{dsca} is able to complete the sweeping task without synchronizing joint actions of moving from the current stripe to the
next stripe. Moreover, the inequality (\ref{updTb}) provides more precise estimation on the upper bound of $\Delta T$ when the partition speed of multi-agent systems along the $y$-axis $v$ is relatively small and the integer $q$ in (\ref{updTb}) is relatively large.
\end{remark}

\begin{theorem}
Collision avoidance of partition bars is guaranteed if the following inequality
\begin{equation}\label{collision}
    \Delta x_i(0)>\kappa T_p \sqrt{\lambda_{max}}\left(2\sqrt{H(0)}+T_p\max_{0\leq\tau\leq T_p}\zeta(\tau)\right),\quad \forall i\in\mathcal{I}_n
\end{equation}
holds with $T_p=(l-\epsilon)/v$, where $\zeta(t)$ is given by equation (\ref{zeta}).
\end{theorem}

\begin{proof}
Let $\Delta x_i(t)=x_i(t)-x_{i-1}(t),~i\in \mathcal{I}_n$ denote the distance between two adjacent partition bars. Then we have
$$
x_i(t)=x_i(0)+\int_{0}^{t}\dot{x}_i(\tau)d\tau=x_i(0)+\kappa\int_{0}^{t}\left(m_{i+1}(\tau)-m_{i}(\tau)\right)d\tau
$$
and
$$
x_{i-1}(t)=x_{i-1}(0)+\int_{0}^{t}\dot{x}_{i-1}(\tau)d\tau=x_{i-1}(0)+\kappa\int_{0}^{t}\left(m_{i}(\tau)-m_{i-1}(\tau)\right)d\tau
$$
The difference between the above two equations leads to
\begin{equation}\label{deltX}
\begin{split}
\Delta x_i(t)&=x_i(t)-x_{i-1}(t)\\
&=x_i(0)-x_{i-1}(0)+\kappa\int_{0}^{t}\left(m_{i+1}(\tau)+m_{i-1}(\tau)-2m_{i}(\tau)\right)d\tau \\
&=\Delta x_i(0)+\kappa\int_{0}^{t}\left(m_{i+1}(\tau)+m_{i-1}(\tau)-2m_{i}(\tau)\right)d\tau
\end{split}
\end{equation}
To simplify the expression, we define
$$
\mathcal{F}_i(t)=\int_{0}^{t}\left(m_{i+1}(\tau)+m_{i-1}(\tau)-2m_{i}(\tau)\right)d\tau
$$
Thus, equation (\ref{deltX}) can be rewritten as
\begin{equation}\label{equalX}
\Delta x_i(t)=\Delta x_i(0)+\kappa\mathcal{F}_i(t)
\end{equation}
and we can obtain
\begin{equation*}
\begin{split}
\mathcal{F}_i(t)&=\int_{0}^{t}\left(m_{i+1}(\tau)+m_{i-1}(\tau)-2m_{i}(\tau)\right)d\tau \\
&=\int_{0}^{t}\left(m_{i+1}(\tau)-m_{i}(\tau)\right)d\tau -\int_{0}^{t}\left(m_{i}(\tau)-m_{i-1}(\tau)\right)d\tau \\
&\geq -\int_{0}^{t}\left|m_{i+1}(\tau)-m_{i}(\tau)\right|d\tau -\int_{0}^{t}\left|m_{i}(\tau)-m_{i-1}(\tau)\right|d\tau \\
&\geq -2t\max_{0\leq\tau\leq t}|m_{i+1}(\tau)-m_{i}(\tau)|\\
\end{split}
\end{equation*}
It follows from Lemma \ref{ineq} and inequality (\ref{dHt1}) that
$$
\max_{0\leq\tau\leq t}|m_{i+1}(\tau)-m_{i}(\tau)|\leq \max_{0\leq\tau\leq t} \left(\sum_{i=1}^{n-1}\left(m_{i+1}(\tau)-m_{i}(\tau)\right)^2\right)^{\frac{1}{2}}\leq\sqrt{\lambda_{max}}\max_{0\leq\tau\leq t}\sqrt{H(\tau)}
$$
and
$$
\max_{0\leq\tau\leq t}\sqrt{H(\tau)}\leq\sqrt{H(0)}+\frac{1}{2}\int_{0}^{t}\zeta(\tau)d\tau
$$
Therefore, we have
\begin{equation}\label{F_ineq}
\begin{split}
\mathcal{F}_i(t)&\geq -2t\max_{0\leq\tau\leq t}|m_{i+1}(\tau)-m_{i}(\tau)|\\
&\geq -2t\sqrt{\lambda_{max}}\left(\sqrt{H(0)}+\frac{1}{2}\int_{0}^{t}\zeta(\tau)d\tau\right)\\
&\geq -T_p\sqrt{\lambda_{max}}\left(2\sqrt{H(0)}+\int_{0}^{T_p}\zeta(\tau)d\tau\right),\quad i\in\mathcal{I}_n
\end{split}
\end{equation}
Substituting inequality (\ref{F_ineq}) into equation (\ref{equalX}) yields
\begin{equation*}
\begin{split}
\Delta x_i(t)&=\Delta x_i(0)+\kappa\mathcal{F}_i(t)\\
&\geq \Delta x_i(0)-\kappa T_p\sqrt{\lambda_{max}}\left(2\sqrt{H(0)}+\int_{0}^{T_p}\zeta(\tau)d\tau\right)\\
&\geq \Delta x_i(0)-\kappa T_p\sqrt{\lambda_{max}}\left(2\sqrt{H(0)}+T_p\max_{0\leq\tau\leq T_p}\zeta(\tau)\right)\\
\end{split}
\end{equation*}
To avoid the collision of partition bars, we need to ensure $\Delta x_i(t)>0,\forall t\geq0,~i\in\mathcal{I}_n$, and it can be achieved by the following inequality
$$
\Delta x_i(0)-\kappa T_p\sqrt{\lambda_{max}}\left(2\sqrt{H(0)}+T_p\max_{0\leq\tau\leq T_p}\zeta(\tau)\right)>0,
$$
which is equivalent to
$$
\Delta x_i(0)>\kappa T_p \sqrt{\lambda_{max}}\left(2\sqrt{H(0)}+T_p\max_{0\leq\tau\leq T_p}\zeta(\tau)\right) \quad \forall i\in\mathcal{I}_n
$$
This completes the proof.
\end{proof}

For rectangular region, we can obtain the following corollary on the upper bound of $\Delta T$.
\begin{coro}
Multi-agent systems adopt the DSCA in Table \ref{dsca} to sweep the rectangular region with the width $l_a$ and the length $l_b$. Then the time error $\Delta T$ is bounded by
$$
\Delta T\leq \frac{1}{\sigma}\sqrt{\frac{(n-1)H(0)}{n}}e^{-\frac{\kappa\underline{\rho}\epsilon\lambda_{min}T_p}{2n}}+\frac{\bar{\rho}(n-1)l_aT_p}{2qn\sigma}\sum_{i=1}^{q}e^{-\frac{\kappa\underline{\rho}\epsilon(q-i)\lambda_{min}T_p}{2nq}}
$$
with $T_p=(l_b-\epsilon)/v$ and $q\in Z^{+}$.
\end{coro}

\begin{proof}
For the rectangular region, we have $g^{'}_a(vt)=g^{'}_b(vt)=0$ and $g_b(y)-g_a(y)=l_a$ , which leads to
$$
\zeta(t)=\left(g_b(vt)-g_a(vt)\right)\bar{\rho}\sqrt{\frac{n-1}{n}}=l_a\bar{\rho}\sqrt{\frac{n-1}{n}}.
$$
Therefore, inequality (\ref{updTb}) can be further simplified as
\begin{equation}
\begin{split}
    \Delta T&\leq \frac{1}{\sigma}\sqrt{\frac{n-1}{n}}\left(\sqrt{H(0)}e^{-\frac{\kappa\underline{\rho}\epsilon\lambda_{min}T_p}{2n}}+\frac{T_p}{2q}\sum_{i=1}^{q}e^{-\frac{\kappa\underline{\rho}\epsilon(q-i)\lambda_{min}T_p}{2nq}}l_a\bar{\rho}\sqrt{\frac{n-1}{n}}\right)\\
    &=\frac{1}{\sigma}\sqrt{\frac{(n-1)H(0)}{n}}e^{-\frac{\kappa\underline{\rho}\epsilon\lambda_{min}T_p}{2n}}+\frac{\bar{\rho}(n-1)l_aT_p}{2qn\sigma}\sum_{i=1}^{q}e^{-\frac{\kappa\underline{\rho}\epsilon(q-i)\lambda_{min}T_p}{2nq}}
\end{split}
\end{equation}
which completes the proof.
\end{proof}

\begin{remark}
For the rectangular region with the width $l_a$ and the length $l_b$, the upper bound of $\Delta T$ approximates
$\frac{l_al_b\bar{\rho}}{2v\sigma}$ when the length of partition bars $\epsilon$ is sufficient small and the number
of agents $n$ is sufficient large.
\end{remark}

\begin{remark}
During the sweeping process, it is assumed that the inequality $\sigma t\leq m_i(t),~\forall t\geq 0, i\in \mathcal{I}_n$ holds, which implies that the sweeping operation lags that of region partition, and the agents are unable to complete the workload in their respective subregion before the region partition comes to an end.
\end{remark}

\section{Numerical Simulations}\label{sec:simu}
This section provides the numerical example to validate our proposed sweep coverage algorithm. Specifically, $5$ mobile agents are instructed to cooperatively sweep the region $\Omega$, which is enclosed by two curves:
$$
x_0=g_a(y)=0.2\sin\frac{\pi(y-4)}{3}+1, \quad x_5=g_b(y)=0.2\sin\frac{\pi(y-4)}{3}+6
$$
and two line segments: $y=0$ and $y=10$. In addition, the distribution density function of workload is described by
$$
\rho(x,y)=\frac{3}{2}+\frac{1}{2}\sin\frac{x^2+y^2}{5}
$$
with the upper bound $\bar{\rho}=2$ and the lower bound $\underline{\rho}=1$. Other parameters are given as follows: $\kappa=1$, $\epsilon=0.01$, $v=8$, $\sigma=6$ and $q=10$. For simplicity, Euler method is employed to implement the partition dynamics (\ref{sys}) with the step size $0.001$. Figure~\ref{sweep} presents the cooperative sweep process of $5$ mobile agents in the irregular region, where the color bar indicates the workload density ranging from light yellow to dark red. At the initial time $t=0$, all the virtual partition bars are located at the bottom of the region. Then multi-agent system starts partitioning the whole region and meanwhile the agents are sweeping their own subregions. At $t=0.5$, part of the region has been partitioned by virtual pars, and each agent completes the same workload $\sigma t=3$ in its own subregion. And the swept parts are marked in white. The task of region partition is finished at $t=1.25$ when all the partition bars arrive at the top of coverage region. Next, the mobile agents continue to sweep their respective subregions. Finally, the sweeping process comes to an end at $t=2.72$ after cleaning the workload in the whole region. Since the optimal sweeping time $T^*$ is equal to $2.54$, we get the time error $\Delta T=0.18$. According to the inequality (\ref{updTb}), the upper bound of $\Delta T$ is $0.88$, which is tighter than the result given by Theorem $4.1$ in \cite{zhai13}.
\begin{figure}\centering
 {\centering\includegraphics[width=0.45\textwidth]{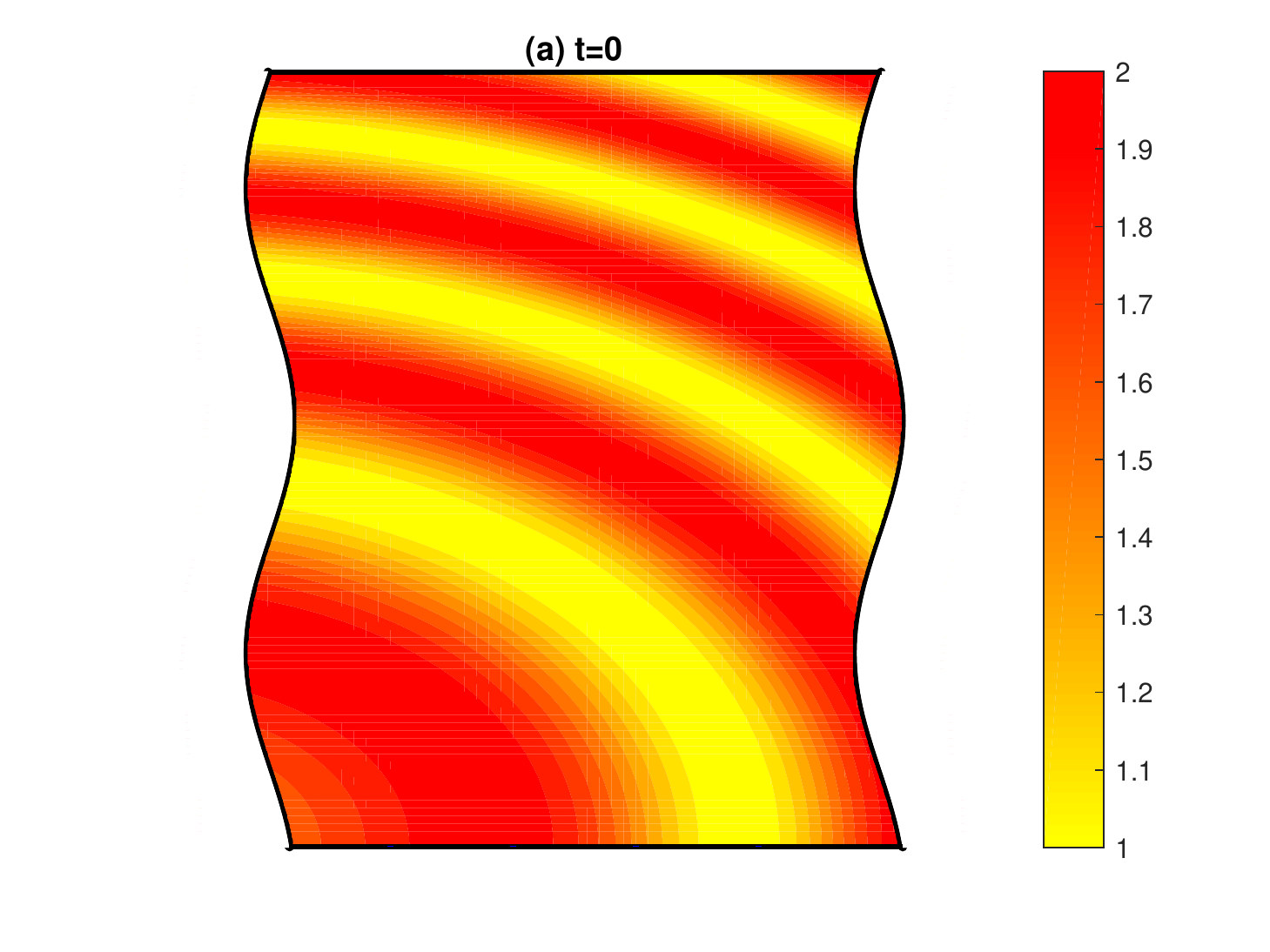}}
 {\centering\includegraphics[width=0.45\textwidth]{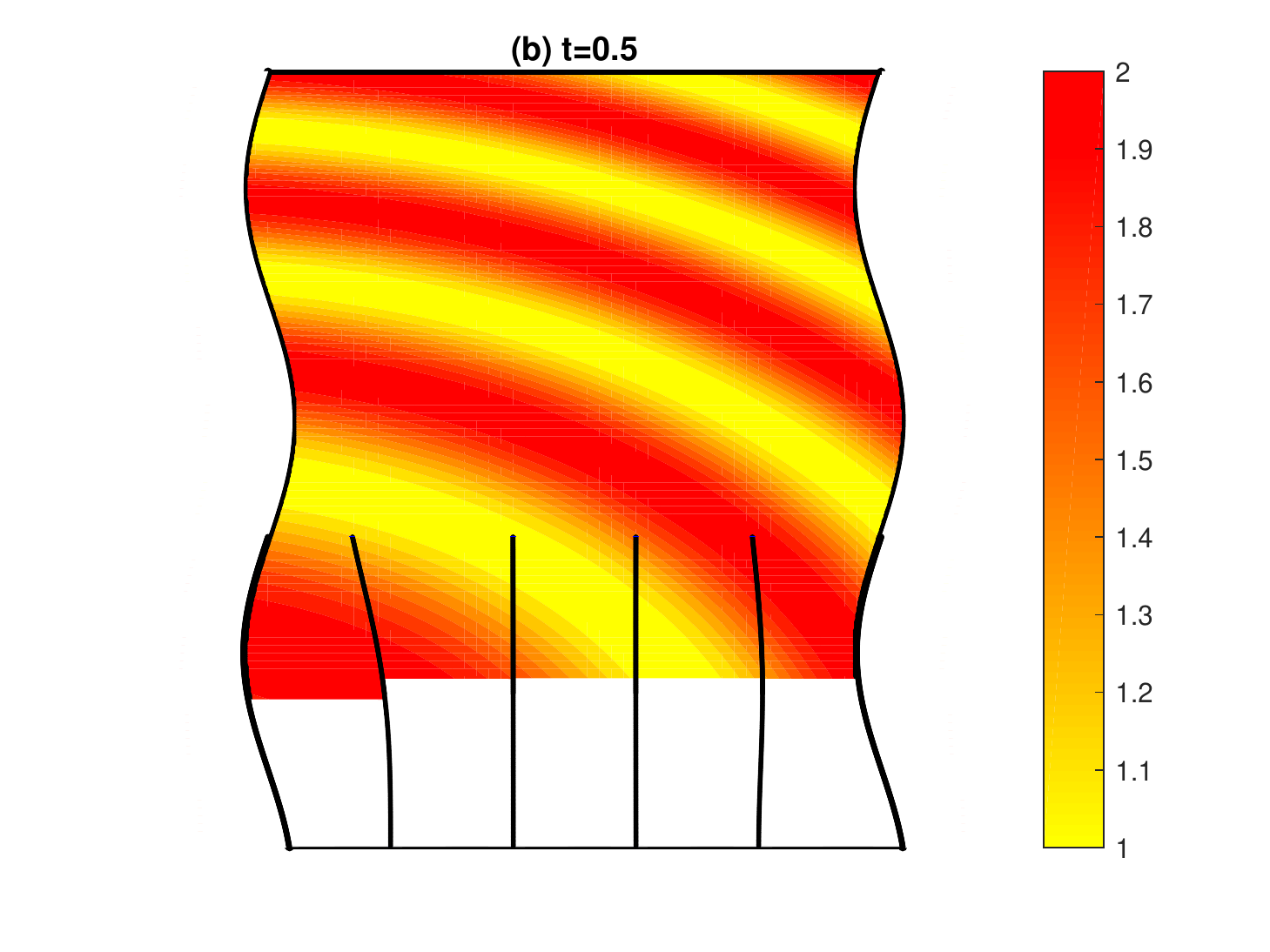}} \\
 {\centering\includegraphics[width=0.45\textwidth]{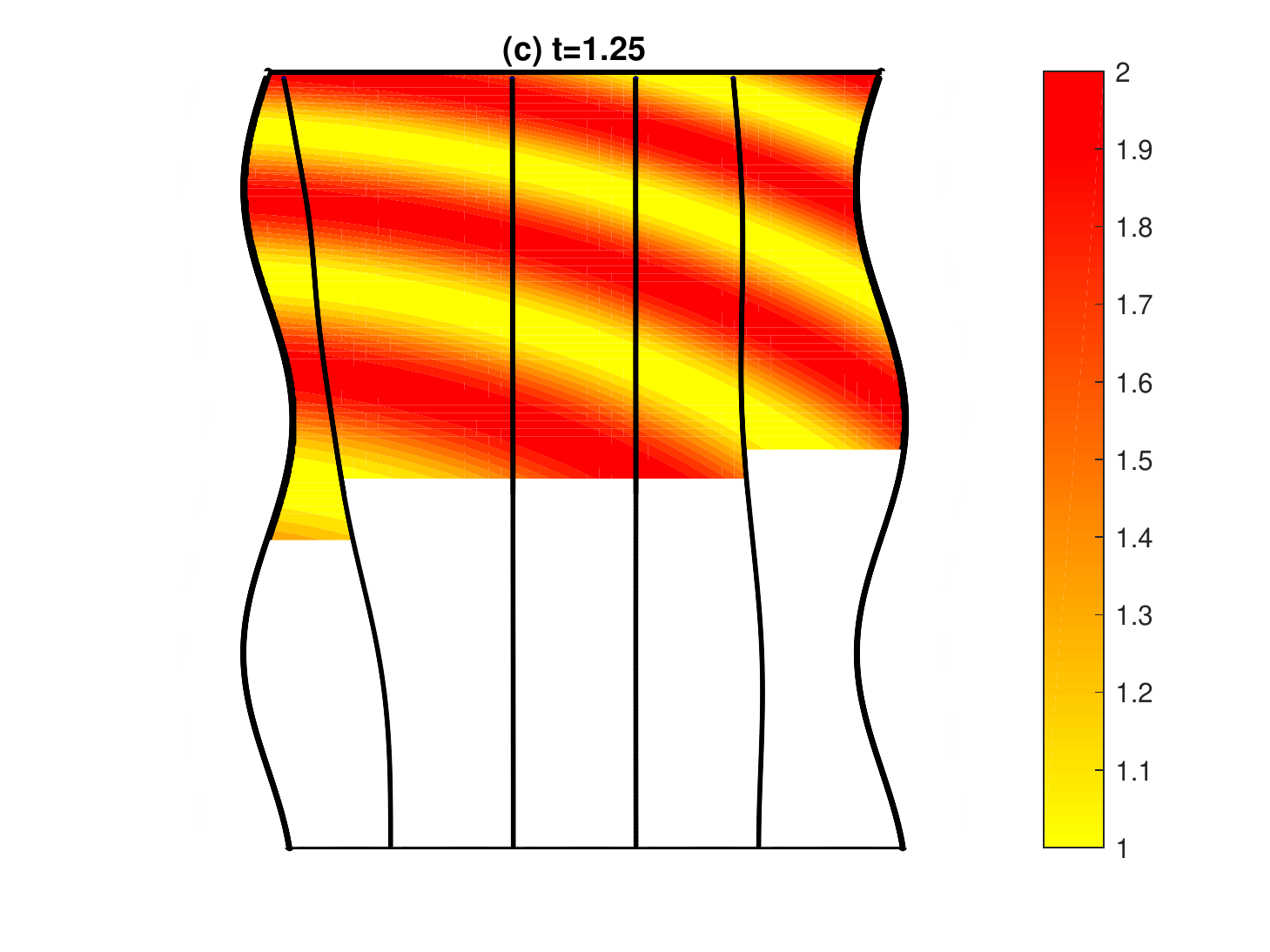}}
 {\centering\includegraphics[width=0.45\textwidth]{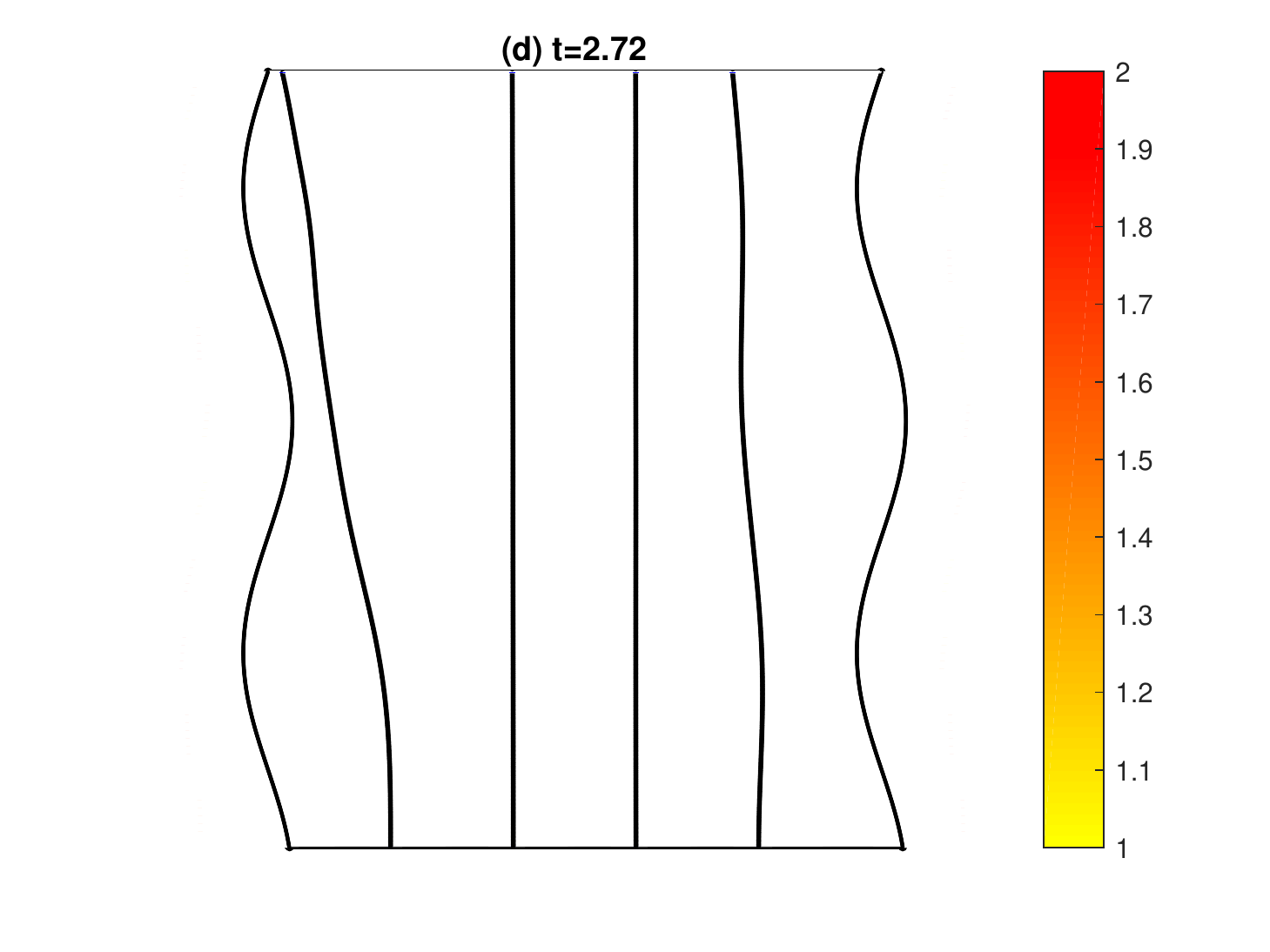}} \\
 \caption{\label{sweep} Sweeping coverage of $5$ mobile agents.}
\end{figure}

\section{Conclusions}\label{sec:con}

In this paper, we developed a novel formulation for the sweep coverage of multi-agent systems in uncertain environment. It has been proved that multi-agent system with the proposed control algorithm is input-to-state stable. Moreover, we obtained the upper bound for the error between the actual sweep time and the optimal time. Numerical simulations demonstrated the effectiveness and advantages of the proposed approach. Future work might include the sweep coverage of uncertain region with convex obstacles and the kinematics of multi-agent system with nonholonomic constraints.


\end{document}